%% file: 0_main.tex
\documentclass[10 pt,conference]{ieeeconf}
\IEEEoverridecommandlockouts
\usepackage{macros}
\usepackage{cite}
\usepackage{dsfont}
\usepackage{amsmath,amssymb,amsfonts}
\usepackage{graphicx}
\usepackage{textcomp}
\usepackage{hyperref}
\usepackage{tabularray}   
\usepackage{adjustbox}

\usepackage{graphics} 
\usepackage{epsfig} 
\usepackage{mathptmx} 
\usepackage{times} 

\newtheorem{theorem}{\textbf{Theorem}}[section]
\newtheorem{lemma}[theorem]{\textbf{Lemma}}
\newtheorem{proposition}[theorem]{\textbf{Proposition}}

\newtheorem{assumption}{\textbf{Assumption}}[section]
\newtheorem{remark}{\textbf{Remark}}[section]

\makeatletter
\def\theorem@qed{\pushQED{\qed}\qedhere\popQED}
\makeatother

\usepackage{xcolor}
\def\BibTeX{{\rm B\kern-.05em{\sc i\kern-.025em b}\kern-.08em
    T\kern-.1667em\lower.7ex\hbox{E}\kern-.125emX}}
    
\begin{document}

\title{ \bf \algoname: A communication-efficient accelerated Newton-type fully distributed optimization algorithm*
}

\author{Souvik Das$^{1}$, Luca Schenato$^2$ and Subhrakanti Dey$^{1}$
\thanks{*This work is funded by Swedish Research Council grant 2023-04232.}
\thanks{$^{1}$Souvik Das and Subhrakanti Dey are with the Department of Electrical Engineering, Uppsala University, Sweden.  Email: {\tt\small \{souvik.das,subhrakanti.dey\}@angstrom.uu.se}}%
\thanks{$^{2}$Luca Schenato is with the Department of Information Engineering, University of Padova, Italy. Email: {\tt\small schenato@dei.unipd.it}}%
}


\maketitle
\begin{abstract} 
This article presents a second-order fully distributed optimization algorithm, \(\algoname\), driven by \emph{heavy-ball} momentum, for \(\lips\)-smooth and \(\strconv\)-strongly convex objective functions. A rigorous convergence analysis is performed, and we demonstrate \emph{global} linear convergence under certain sufficient conditions. Through extensive numerical experiments, we show that \(\algoname\) with heavy-ball momentum achieves acceleration, and the corresponding rate of convergence is strictly faster than its non-accelerated version, NETWORK-GIANT. Moreover, we compare \(\algoname\) with several state-of-the-art algorithms, both momentum-based and without momentum, and report significant performance improvement in convergence to the optimum. We believe that this work lays the groundwork for a broader class of second-order Newton-type algorithms \emph{with momentum} and motivates further investigation into open problems, including an analytical proof of local acceleration in the fully distributed setting for convex optimization problems.
\end{abstract}

\input{1_Introduction}
\input{2_PS}

\input{3_Results}
\input{4_Experiments}

\section{Conclusion}
\label{sec:conclusion}
This article presents a heavy-ball momentum-enhanced enhanced fully distributed approximate Newton-type optimization algorithm. We establish sufficient conditions for linear convergence, governed by the regularity of the objective functions and the connectivity of the underlying graph. Numerical experiments demonstrate superior performance compared to their existing first-order and non-accelerated second-order (Network-GIANT) counterparts. 

This work lays the foundation for a more general class of \emph{second-order Newton-type} algorithms \emph{with momentum}, and motivates us to investigate several open questions that exist to date, one of which includes analytically proving \emph{local} acceleration in the fully distributed regime, for smooth, strongly convex as well as smooth, non-strongly convex optimization problems.

\appendices
\input{Appendix}

\bibliographystyle{IEEEtran}
\bibliography{refs}

\end{document}

%% file: 1_Introduction.tex
\section{Introduction}
\label{sec:intro}

Decentralized/Distributed optimization has emerged as an attractive paradigm to  effectively solve large-scale optimization problems with reasonable accuracy, in various domains such as signal processing and sparse linear regression \cite{ref:XH-XH-XJ-23}, machine learning \cite{ref:HTW-ZY-ZW-MH-18,ref:AN-20} and risk minimization \cite{ref:KK-HL-20}, control systems \cite{ref:IN-VN-13,ref:AN-JL-18}, large scale sensor networks \cite{ref:MR-RN-04}, and power grids \cite{ref:QL-XL-HL-TH-20}.

Early works in distributed optimization for convex objective functions date back to \(1950\) \cite{ref:MRH-ES}. It has garnered renewed interest due to its low per-iteration cost and ease of deployment in a distributed setting. While distributed or decentralized optimization can also refer to a setting where nodes can communicate to a central server (e.g., Federated Learning (FL)), in this paper, we will only focus on a fully distributed setting without a central server, where nodes can only communicate to their immediate neighbours. In this context, distributed optimization algorithms are primarily based on distributed gradient descent and distributed dual averaging methods, applied to undirected and directed graphs. The authors in \cite{ref:JCD-AA-MJW-11,ref:AN-AO-14} reported sublinear convergence for convex objectives due to the use of a diminishing step-size. Starting with the initial work \cite{EXTRA_algo} which introduced the idea of gradient tracking,  the works \cite{ref:GQ-NL-17,ref:AN-AO-WS-17,ref:RX-UAK} have established linear convergence for strongly convex functions with a constant step-size, by exactly estimating the global gradient and attaining convergence to the exact optima. It is well known that accelerated convergence can be achieved by momentum-based or accelerated gradient descent methods \cite{ref:BTP-64,ref:YN-83}, in the centralized optimization paradigm. In the distributed setting with gradient descent-based methods, 
 both heavy-ball and Nesterov's acceleration methods have been implemented, and global linear convergence has been reported under sufficient conditions \cite{ref:RX-UAK-19,ref:GQ-NL-19, ref:JG-XL-YHD-YH-PY-22}. Additional aspects like directed graphs \cite{ref:RX-DJ-UAK-19,ref:JG-XL-YHD-YH-PY-22 }, time-varying graphs \cite{ref:DTYN-DTN-AN-23}, stochastic gradient tracking \cite{ref:SP-AN-21}, and parametric varying momentum-based algorithms \cite{ref:SH-JG-QLD-CZ-25} have been also been investigated in this context. However, only a few articles discuss provable accelerated convergence, e.g.\cite{ref:RX-UAK-SK-20,ref:BL-LC-JY-22}. For the non-strongly convex case, the authors in \cite{ref:YL-WL-BZ-JD-24} reported sublinear convergence with a vanishing step-size. 
 
 Although Newton’s method is known for its fast, locally quadratic convergence, its adoption in distributed optimization has been limited due to the high cost of Hessian inversion and the communication overhead of sharing local Hessians among nodes. In recent years, various distributed approximate Newton-type methods have emerged, which include, e.g, networked Newton-Raphson techniques (requiring full Hessian communication) \cite{varagnolo2015newton}, Network-Newton type methods using a truncated Taylor series expansion of the inverse Hessian \cite{mokhtari2016network}, and Network-DANE
 \cite{li2020communication}, based on an approximate Newton-type FL algorithm called DANE 
 \cite{shamir2014communication}. More recently, a new approximate fully distributed Newton-type algorithm, called Network-GIANT \cite{ref:AM-GS-LS-SD-23}  (based on the FL algorithm GIANT \cite{ref:SW-FR-PX-MWM-18}), which was shown to have the same communication complexity as distributed gradient-tracking based methods, with global linear convergence under appropriate conditions on the step-size, and faster convergence than many of the existing distributed second-order methods mentioned earlier. The incorporation of acceleration into approximate Newton-type algorithms remains relatively unexplored. Notable exceptions include \cite{ref:HY-LL-ZZ-20}, where the authors developed a centralized Newton-type algorithm incorporating Nesterov’s acceleration and reported improved local convergence, and \cite{ref:TZY-LP-20}, where DANE is combined with heavy-ball momentum in the federated learning regime. However, in the fully distributed setting, the design and analysis of accelerated, communication-efficient Newton-type methods remain unexplored. We address this gap by introducing an accelerated version of Network-GIANT, termed \(\algoname\),
 which adopts a heavy-ball momentum scheme.

\subsection*{Our contributions}
\begin{itemize}[leftmargin=*, label = \(\circ\)]
    \item We present a fully distributed approximate Newton-type optimization algorithm with \emph{heavy-ball} acceleration for a group of agents connected over an undirected network. The algorithm has a communication complexity of \(\Bigoh{(\agents \dimdv)}\), where \(\agents\) and \(\dimdv\) denote the number of agents and the dimension of the decision space, respectively, identical to that of gradient-based algorithms.
    \item We establish a \emph{global linear convergence rate} of the order \(\Bigoh{\Big(\spec(\conmat(\stsz, \hbstsz))^t\Big)}\) for objective functions that are \(\lips\)-smooth and \(\strconv\)-strongly convex, under certain verifiable and mild sufficient conditions. Here \(t\) is the iteration index and \(\spec(\conmat(\stsz, \hbstsz))\) refers to the spectral radius of the matrix describing the dynamics of the errors; see Proposition \ref{prop:lin ineq} and Theorem \ref{th:key} for further information.

    \item We evaluate the performance of \(\algoname\) against several state-of-the-art accelerated gradient based distributed algorithms  through extensive numerical experiments, including both heavy-ball momentum and Nesterov's accelerated methods. Our results empirically demonstrate that \(\algoname\) indeed achieves an accelerated convergence rate due to the addition of the heavy-ball momentum, in comparison to its unaccelerated counterpart, Network-GIANT, and 
    first-order distributed accelerated methods \cite{ref:RX-UAK-19,ref:GQ-NL-19}.
\end{itemize}

\subsection*{Organization}The article is organized as follows: \S \ref{sec:ps} presents the problem formulation and the underlying assumptions, \S \ref{sec:algo} presents the key steps of \(\algoname\) and its associated notations, the main features of \(\algoname\), and its convergence analysis. The main theoretical results --- Proposition \ref{prop:lin ineq} and Theorem \ref{th:key} --- of this article are given in \S \ref{subsec:convergence}. Numerical experiments are included in \S \ref{sec:experiments},  followed by the concluding remarks in \S \ref{sec:conclusion}.

\subsection*{Notations}
We employ the following notations. For \(D \in \Rbb^{\agents \times \agents}\), \(\spec(D)\) is its spectral radius.\footnote{Here the norm is induced \(2\)-norm of a matrix, which is defined as\(\norm{D}_2 = \sup_{x \neq 0, x \in \Rbb^n} \frac{\norm{Ax}_2}{\norm{x}_2}\).} All vector norms are assumed to be Euclidean, and matrix norms \(\norm{\cdot}:\Rbb^{m \times n}\lra \Rbb\) to be the Frobenius norms, unless specified otherwise. For any square matrix, the symbol \(\norm{\cdot}_2\) denotes the spectral norm (or the induced \(2\)-norm). Moreover, for a square matrix \(D\) and any \(X \in \Rbb^{n \times p}\), the identity \(\norm{DX}_F \le \norm{D}_2 \norm{X}_F\) is used in the analysis of the algorithm, and for simplicity, we drop the subscript $F$ from the matrix Frobenius norms henceforth.

%% file: 2_PS.tex
\section{Problem formulation}
\label{sec:ps}
\subsection{Problem description}\label{subsec:ps}
We study the distributed optimization problem 
\begin{equation}\label{eq:key prob}
    \min_{\dvar \in \Rbb^{\dimdv}} \objfunc(\dvar) \Let \frac{1}{\agents}\sum_{i=1}^{\agents} \objfunc_i(\dvar),
\end{equation}
consisting of  \(\agents\)-agents connected over a network \(\network\). Here \(\dvar \in \Rbb^{\dimdv}\) is the decision variable, for each \(i\)-th agent, where \(i \in \aset[]{1, 2, \ldots, \agents}\) the function \(\objfunc_i: \Rbb^{\dimdv} \mapsto \Rbb\) denotes the local objective function, which depends on the local data \(\ldata_i\). The goal of the distributed optimization problem is for all nodes to collaboratively reach the global minimizer using only local computations and communications with neighboring nodes.


The underlying communication network is modeled as an \emph{undirected} graph \(\network = \aset[\big]{\nodeset, \edgeset}\), where \(\nodeset = \aset[]{1,2,\cdots,\agents}\) denotes the set of nodes (or vertices) representing the agents, and \(\edgeset \subset \nodeset \times \nodeset\) denotes the set of edges. Each edge \((i,j) \in \edgeset\) represents an undirected link between agents \(i\) and \(j\), allowing them to communicate and exchange data directly. For each agent \(i \in \nodeset\), we define the set of \emph{neighbors} by \(\ngbr_i \Let \aset[]{j \in \nodeset \suchthat (i,j) \in \edgeset}\). 

\subsection{Standing assumptions}\label{subsec:sa}
\begin{assumption}[Consensus weight matrix]
    \label{assum:on graph}
    A nonnegative doubly stochastic consensus weight matrix \(\Wght = [\wght_{ij}]_{i,j=1}^{\agents}\), satisfying \(\sum_{i'} \wght_{i'j} = \sum_{j'} \wght_{ij'}=1\) for every \(i',j' \in \nodeset\), encodes the network structure. Agent \(i\) can receive data from agent \(j\) if and only if \(\wght_{ij}>0\), and \(\wght_{ii}>0\) for \(i \in \nodeset\).
\end{assumption}

\begin{assumption}[Regularity of local objective functions]
    \label{assum:Standard assumptions}
    We impose the standing assumptions: for each \(i \in \aset[]{1,2,\ldots,\agents}\), \(\objfunc_i: \Rbb^{\dimdv} \mapsto \Rbb\) is twice differentiable, \(\strconv\)-strongly convex, and \(\lips\)-smooth, i.e., there exist positive constants \(\strconv,\ \lips\) such that for each \(i \in \aset[]{1,2,\ldots, \agents}\) and for all \(\dvar,y \in  \Rbb^{\dimdv}\), we get
    \begin{equation}\label{eq:con_smoothness}
        \begin{cases}
            \hess \objfunc_i (\dvar) \succeq \strconv \identity{\dimdv}\\
            \norm{\grd \objfunc_i(\dvar) - \grd \objfunc_i(y)} \leq \lips \norm{\dvar - y}.
        \end{cases}
    \end{equation}
\end{assumption}
The above two conditions in \eqref{eq:con_smoothness} imply that 
\begin{equation}
    \label{eq:strongly_convex_Lips}
    \strconv \identity{\dimdv} \preceq \hess \objfunc_i(\dvar) \preceq \lips \identity{\dimdv} \text{ for every }\dvar \in \Rbb^{\dimdv},
\end{equation}
for each \(i \in \aset[]{1,2,\ldots, \agents}\). It also follows immediately that the global objective function $f(x)$ is strongly convex and Lipschitz continuous with the same parameters and $\mu \identity{\dimdv} \preceq \nabla^2 f(x) \preceq \lips \identity{\dimdv}$. The ratio 
\(\cn \Let \frac{\lips}{\strconv}>1,\)
refers to the condition number associated with the Hessian of the local objective functions \((\objfunc_i)_{i=1}^{\agents}\). Note that the optimization problem \eqref{eq:key prob} admits an optimal solution \(\dvar^{\ast}\), which follows as an immediate consequence of strong convexity \cite[p. no. \(460\)]{ref:SB-LV-04}.

%% file: 3_Results.tex
\section{\(\algoname\) and its analysis}
\label{sec:algo}

\subsection{Key steps of \(\algoname\)}\label{subsec:algo}
In \(\algoname\), the key idea of Network-GIANT \cite{ref:AM-GS-LS-SD-23} is extended by including a \emph{heavy-ball momentum term}, driven by a gradient-tracking mechanism. The key consensus steps of \(\algoname\) is described by the following recursions at each \(i\)-th agent, 
\begin{equation}\label{eq:agent con step}
\begin{aligned}
    &\dvar_i(t+1) = \sum_{j \in \ngbr_i \cup \aset[]{i}} \wght_{ij}\dvar_j(t) - \stsz \cdvar_i(t) + \hbstsz \hbvar_i(t), \\
    &\gt_i(t+1) = \grd \objfunc_i\big(\dvar_i(t+1)\big) - \grd \objfunc_i\big(\dvar_i(t)\big), \\
    &\gtvar_i(t+1) = \sum_{j \in \ngbr_i \cup \aset[]{i}} \wght_{ij} \gtvar_j(t) + \gt_i(t+1),\\
    & \hbvar_i(t+1) = \dvar_i(t+1) - \dvar_i(t), \text{ and }\cdvar_i(t) = \big[\hess \objfunc_i(\dvar_i(t))\big]{\inverse} \gtvar_i(t),
\end{aligned}
\end{equation}
for each iteration \(t \in \Nz\), where \(\stsz, \hbstsz, > 0 \) are scalars to be picked appropriately, and \(\gtvar_i(0) = \grd \objfunc_i\big(\dvar_i(0)\big)\) must be satisfied. Here the variable \(\gtvar(t)\) for each \(i\), refers to as  a gradient tracker \cite{ref:GQ-NL-17}. For each \(i \in \aset[]{1,2,\ldots, \agents}\), the variable \(\cdvar_i(t) \Let \big[\hess \objfunc_i(\dvar_i(t))\big]{\inverse} \gtvar_i(t) \in \Rbb^{\dimdv}\) represents the local descent direction, and \(\gt_i(t)\) with \(\gt_i(0) =\grd\objfunc_i\big(\dvar_i(0)\big) \in \Rbb^{\dimdv}\) keeps track of the history of the local gradient at each \(t\), and \(\hbvar(t)\) with \(\hbvar(0) = 0\) is the momentum term.   Here \(\dvar_i(t) \in \Rbb^{\dimdv}\) and \(\gtvar_i(t) \in \Rbb^{\dimdv}\) denote the local estimates of global variables \(\dvar(t) \in \Rbb^{\dimdv}\) and \(\gtvar(t) \in \Rbb^{\dimdv}\) of \(\dvar(t) \in \Rbb^{\dimdv}\) and \(\gtvar(t) \in \Rbb^{\dimdv}\)  at the \(t\)-th iteration, respectively. 

We adopt a matrix notation to compactly represent \eqref{eq:agent con step}. To that end, we define the global decision variables  \(\dmat(t) \Let \bigl(\dvar_1(t)\; \cdots\; \dvar_{\agents}(t) \bigr)^{\top} \), \(\gtmat(t) \Let \bigl(\gtvar_1(t)\; \cdots \; \gtvar_{\agents}(t) \bigr)^{\top} \), \(\cdmat(t) \equiv \cdmat\big(\dmat(t) \big) \Let \big( \cdvar_1(t) \cdots \cdvar_{\agents}(t)\big)^{\top} \), \(\G(t) \Let \big(\gt_1(t) \cdots \gt_{\agents}(t) \big)^{\top}\), and \(\hbmat(t) \Let \big(\hbvar_1(t) \cdots \hbvar_{\agents}(t) \big)^{\top}\) all of them belong to \( \in \Rbb^{\agents \times \dimdv}\).
Let \(\Objfunc: \Rbb^{\agents \times \dimdv} \lra \Rbb\) be the aggregate objective function, defined by \(\Objfunc(\dmat(t)) \Let \frac{1}{n}\sum_{i=1}^{\agents}\objfunc_i(\dvar_i(t))\), and let \(\grd \Objfunc(\dmat(t)) \Let \bigl(\grd \objfunc_1(\dvar_1(t)) \; \cdots \; \grd \objfunc_{\agents}(\dvar_{\agents}(t)) \bigr)^{\top} \in \Rbb^{\agents \times \dimdv}\). Then \(\grd \ol{\Objfunc}(\dmat(t)) \Let \frac{1}{\agents} \ones^{\top} \grd \Objfunc(\dmat(t))\). We compactly write the consensus steps \eqref{eq:agent con step}: 
\begin{equation}\label{eq:update rule}
    \begin{aligned}
        &\dmat(t+1) = \Wght\dmat(t) -   \stsz \cdmat(t) + \hbstsz \hbmat(t),\\
        & \G(t+1) = \grd \Objfunc \big(\dmat(t+1) \big) - \grd \Objfunc \big(\dmat(t) \big), \\
        &\gtmat(t+1) = \Wght\gtmat(t)  + \G(t+1),\\
        &\hbmat(t+1) = \dmat(t+1) - \dmat(t), \text{ for each }t \in \Nz,
    \end{aligned}
\end{equation}
with \(\gtmat(0) = \grd \Objfunc\big(\dmat(0)\big)\), \(\G(0) =\gtmat(0)\), and \(\hbmat(0) = \dmat(0) - \dmat(-1)\) with \(\dmat(0) = \dmat(-1)\) . The quantity \(\stsz>0\) is the step-size of \(\algoname\), and \(\hbstsz>0\) is the momentum coefficient that controls the momentum term \(\hbstsz \hbmat(t)\).

\subsection{Features of \(\algoname\)}
\label{subsec:features }
\(\algoname\) is built on three key principles:
\begin{itemize}[leftmargin=*, label = \(\circ\)]
    \item \embf{Gradient tracking:} The second consensus step in \eqref{eq:update rule}, asymptotically tracks the past gradients with \emph{zero error} and is essential for convergence to the exact optimal solution. Further details are provided in Lemma \ref{it:port_1}.

    \item \embf{Second-order oracle:} 
    Approximate Newton-type methods are known to achieve faster convergence than first-order gradient-based algorithms. We adopt the key idea of Network-GIANT, wherein, for each iteration \(t\), the sequence \((\dvar_i(t), \gtvar_i(t))_{i=1}^{\agents}\), are exchanged among the agents instead of the local Hessians, reducing the per-iteration communication complexity to \(\Bigoh{\big( \agents \dimdv\big)}\) only.

    \item \embf{Heavy-ball acceleration:} The third component of \(\algoname\) is the heavy-ball momentum term, denoted by \(\hbstsz \hbmat(\cdot)\), which is well-known to accelerate the rate of convergence in centralized gradient descent algorithms \cite{ref:BTP-64}. It was proved that for specific choice of the step-sizes and the momentum coefficient, the method guarantees faster convergence locally for quadratic functions with a local rate given by \(\Bigoh{\Big(\frac{\sqrt{\cn} - 1}{\sqrt{\cn} + 1} \Big)^t}\), where \(\cn\) is the condition number. The convergence analysis done in this article establishes \emph{global} linear convergence under certain sufficient conditions. However, we demonstrate accelerated convergence only through numerical experiments.
        
   
\end{itemize}
\input{3_the_algorithm}
\subsection{Convergence of \(\algoname\)}
\label{subsec:convergence}
We establish linear convergence of \(\algoname\) in this section. The proof methodology involves establishing a linear system of inequalities and identifying a set of sufficient conditions on the step-size \(\stsz\) and the momentum coefficient \(\hbstsz\). Define the errors: 
\begin{itemize}[leftmargin=*, label = \(\circ\)]
    \item consensus error \(\conerr(t) \Let \norm{\dmat(t) - \ones \ol{\dmat}(t)}\),
    \item gradient tracking error \(\gterr(t) \Let \norm{\gtmat(t) - \ones\avggtmat(t)}\),
    \item optimization error \(\opterr(t) \Let \sqrt{\agents}\norm{\avgdmat(t) - \dvar^{\ast}},\)
    \item the momentum term \(\hberr(t) \Let \norm{\dmat(t) - \dmat(t-1)}\).
\end{itemize}
We next present several properties of the variables of interest: the gradient-tracking variable \(\gtmat\) and the decision variable \(\dmat\).
\begin{lemma}
    \label{lem:portmanteuau 1}
    Consider Algorithm \ref{alg:sec_ord_comp} along with its associated notations. Suppose that Assumptions \ref{assum:on graph} and \ref{assum:Standard assumptions} hold. Then the following assertions hold:
    \begin{enumerate}[leftmargin=*,label = (\ref{lem:portmanteuau 1}-\alph*)]
        \item \label{it:port_1}If the initial condition \(\gtmat(0)\) is chosen such that \(\gtmat(0) = \grd \Objfunc(\dmat(0))\), then for every \(t\ge 0\) we have \(\avggtmat(t) = \grd \ol{\Objfunc}(\dmat(t)) \Let \frac{1}{\agents} \ones^{\top} \grd \Objfunc(\dmat(t))\).

        \item \label{it:port_2} For every \(t \ge 0\), we get
        \(\avgdmat(t+1) = \avgdmat(t) - \stsz \frac{1}{\agents}\ones^{\top} \cdmat(t) + \hbstsz \frac{1}{\agents}\ones^{\top}
        \hbmat(t).\)

        \item \label{it:port_3} Let \(L\) be the Lipschitz constant in Assumption \ref{assum:Standard assumptions}. For every \(t \ge 0\), \(\norm{\avggtmat(t)^{\top} - \grd \objfunc\bigl( \avgdmat(t)^{\top}\bigr) } \le \frac{\lips}{\sqrt{\agents}}     \norm{\dmat(t) - \ones \avgdmat(t)}\) and \(\norm{\grd \Objfunc\big(\dmat(t+1) \big) - \grd \Objfunc\big(\dmat(t) \big)} \le \lips \norm{\dmat(t+1) - \dmat(t)}\).

        \item \label{it:port_5} The inequality \(\norm{\grd \av{\Objfunc} \big(\ones \avgdmat(t) \big) - \grd \objfunc\big(\dvar^{\ast}\big)} \le \lips \norm{\avgdmat(t)^{\top} - \dvar^{\ast}}\) is valid for each \(t\).

        \item \label{it:port_4} Recall that \(\strconv>0\) is the constant associated with strong convexity of the sequence of functions \((\objfunc_i)_{i=1}^{\agents}\) in Assumption \ref{assum:Standard assumptions}. For each \(t \ge 0\), we have \(\norm{\cdmat(t)  - \ones \avgcdmat(t)} \le \frac{1}{\strconv}\norm{\gtmat(t) - \ones \grd \ol{\Objfunc}(\dmat(t))} + \frac{\lips}{\strconv} \norm{\dmat(t) - \ones \avgdmat(t)} + \sqrt{\agents} \frac{\lips}{\strconv} \norm{\avgdmat(t)^{\top} - \dvar^{\ast}}\).
        
        \item \label{it:port_6} The momentum term \(\hbmat(t)\) at each iteration \(t\) satisfies \(\norm{\hbmat(t) - \frac{1}{\agents} \ones \ones^{\top} \hbmat(t)} \le \norm{\hbmat(t)}\).
    \end{enumerate}
\end{lemma}
A proof of Lemma \ref{lem:portmanteuau 1} is deferred to Appendix \ref{sec:proof of lemma}.
The next result establishes a set of linear inequalities satisfied by the augmented error \(\errvec \Let (\conerr, \gterr, \opterr,\hberr)\):
\begin{proposition}
    \label{prop:lin ineq}
    Consider Algorithm \ref{alg:sec_ord_comp} with its notations, and suppose that Assumptions \ref{assum:on graph} and \ref{assum:Standard assumptions} hold. Define  the quantity
    \(\quanto \Let \norm{\Wght - \identity{\dimdv}}_2 \text{ and }\specnorm \Let \norm{\Wght   - \frac{1}{\agents}\ones \ones^{\top}}_2 < 1,\)
    with \(\Wght \in \Rbb^{\agents \times \agents}\) be the consensus weight matrix. Recall that \(\stsz>0\) and \(\hbstsz>0\) denote the step-size and the momentum coefficient in \eqref{eq:update rule}, and \(\cn\) refers to the condition number of the local Hessian of \(\objfunc_i\) for each \(i\). If \(\stsz \le \frac{\strconv}{\lips}\), then for each \(t \in \Nz\), \(\errvec(t+1) \le \conmat(\stsz,\hbstsz)\errvec(t)\), where 
    \begin{align}
        \label{eq:lin system ineq}
        \conmat(\stsz,\hbstsz) \Let \begin{pmatrix}
            \specnorm + \stsz \cn & \frac{\stsz}{\strconv} & \stsz \cn & \hbstsz\\
            \lips (\quanto + \stsz \cn) & \specnorm + \stsz \cn & \lips \stsz \cn & \lips 
            \hbstsz\\
            \stsz \cn & \frac{\stsz}{\strconv} & 1 - \frac{\stsz}{\cn} & \hbstsz\\
            \quanto + \stsz \cn & \frac{\stsz}{\strconv} & \stsz \cn & \hbstsz
        \end{pmatrix},
    \end{align}
    \(\strconv>0\) and \(\lips>0\) are the strong-convexity parameter and Lipschitz constant of the local gradients. 
\end{proposition}
\begin{proof}
    The proof proceeds by deriving an upper bound on the errors in terms of the other errors. We start with the consensus error. Observe that
    \begin{align}\label{eq:con err}
        &\norm{\dmat(t+1) - \ones \avgdmat(t+1)} \\
        & = \Bigg \lVert \Big(\Wght \dmat(t) - \ones \avgdmat(t) \Big) + \stsz \bigg(\frac{1}{\agents}  \ones \ones^{\top} \cdmat(t) - \cdmat(t)\bigg)\nn \\
        & \hspace{1cm} + \hbstsz \bigg(\hbmat(t) - \frac{1}{\agents} \ones \ones^{\top} \hbmat(t) \bigg) \Bigg \rVert \nn\\
        & \le \specnorm \norm{\dmat(t) - \ones \avgdmat(t)} + \stsz \underbrace{\norm{\cdmat(t) - \ones \avgcdmat(t)}}_{T_1} + \hbstsz \underbrace{\norm{\hbmat(t) - \ones \avghbmat(t)}}_{T_2}.\nn
    \end{align}
    In the last inequality, we use the identity 
    \begin{align}\label{eq:submultiplicity}
        &\norm{(\Wght \dmat(t) - \ones \avgdmat(t)}\\
        & = \norm{\bigg(\Wght - \frac{1}{\agents}\ones \ones^{\top} \bigg)\big( \dmat(t) - \ones\avgdmat(t)\big)} \nn \\
        & \le \norm{\Wght - \frac{1}{\agents}\ones \ones^{\top}}_2 \norm{\dmat(t) - \ones\avgdmat(t)} \le \specnorm \norm{\dmat(t) - \ones\avgdmat(t)}.\nn
    \end{align}
    Invoking Lemma \ref{it:port_4} and \ref{it:port_6} to get an upper bound of the terms \(T_1\) and \(T_2\), the resultant expression of \eqref{eq:con err} is 
    \begin{align}
        \label{eq:con err result}
        &\norm{\dmat(t+1) - \ones \avgdmat(t+1)} \le \big(\specnorm + \stsz \cn \big) \norm{\dmat(t) - \ones \avgdmat(t)}  \nn\\
        & \hspace{1.5mm}\frac{\stsz}{\strconv} \norm{\gtmat(t) - \ones \grd \ol{\Objfunc}(\dmat(t))} + \stsz \cn \sqrt{\agents} \norm{\avgdmat(t)^{\top} - \dvar^{\ast}}\nn \\& \hspace{5cm} + \hbstsz \norm{\hbmat(t)}.
    \end{align}
    Let us now focus on the gradient-tracking error. Note that
    \begin{align}\label{eq:gt aux_1}
        &\norm{\gtmat(t+1) - \ones \avggtmat(t+1)}^2 \nn \\
        & = \norm{\gtmat(t+1) - \ones \avggtmat(t) + \ones \avggtmat(t) - \ones \avggtmat(t+1)}^2 \nn \\
        & = \norm{\gtmat(t+1) - \ones \avggtmat(t)}^2 + \norm{\ones \big(\avggtmat(t) -  \avggtmat(t+1)\big)}^2 \nn \\
        & \hspace{1cm} - 2 \agents \inprod{\frac{1}{\agents}\sum_{i=1}^{\agents} \gtvar_i(t+1) - \avggtmat(t)}{\avggtmat(t+1) - \avggtmat(t)} \nn \\
        & = \norm{\gtmat(t+1) - \ones \avggtmat(t)}^2 + \agents \norm{\avggtmat(t) -  \avggtmat(t+1)}^2 \nn \\
        & \hspace{0.5cm} - 2 \agents \norm{\avggtmat(t) -  \avggtmat(t+1)}^2 \le \norm{\gtmat(t+1) - \ones \avggtmat(t)}^2,
    \end{align}
    where 
    \begin{align}\label{eq:gt aux_2}
        &\norm{\gtmat(t+1) - \ones \avggtmat(t)} \nn \\ & \le  \norm{\Wght \gtmat(t) - \ones \avggtmat(t) + \grd \Objfunc\big(\dmat(t+1) \big) - \grd \Objfunc\big(\dmat(t) \big) }\nn \\
        & \le \specnorm \norm{\gtmat(t) - \ones \avggtmat(t)} + \norm{\grd \Objfunc\big(\dmat(t+1) \big) - \grd \Objfunc\big(\dmat(t) \big)} \nn \\
        & \le \specnorm \norm{\gtmat(t) - \ones \avggtmat(t)} + \lips \norm{\dmat(t+1) - \dmat(t)} \nn \\
        & \qquad \text{(from Lemma \ref{it:port_3})}\nn \\
        & = \specnorm \norm{\gtmat(t) - \ones \avggtmat(t)} + \lips \norm{\hbmat(t+1)}.
    \end{align}
    We restrict our attention to \(\norm{\hbmat(t+1)}\) for each \(t\). Recall that \(\quanto = \norm{\Wght - \identity{\dimdv}}_2\). Expanding the recursion for the decision variable in \eqref{eq:update rule}, and using \((\identity{\dimdv} - \Wght)\ones = 0\), we get
    \begin{align}\label{eq:mom term}
        &\norm{\hbmat(t+1)} = \norm{\Wght \dmat(t) - \stsz \cdmat(t) + \hbstsz \hbmat(t) - \dmat(t)} \nn \\
        & \le \norm{\Wght \dmat(t) - \stsz \cdmat(t)  - \dmat(t)} + \norm{\hbstsz \hbmat(t)} \nn \\
        & \le \norm{\Wght \dmat(t) + (\identity{\dimdv} - \Wght)\ones \avgdmat(t) - \stsz \cdmat(t)  - \dmat(t)} + \norm{\hbstsz \hbmat(t)}  \nn \\
        & = \norm{\Wght \big(\dmat(t) - \ones \avgdmat(t) \big) - \identity{\dimdv} \big(\dmat(t) - \ones \avgdmat(t) \big) - \stsz \hbmat(t)} \nn \\
        & \hspace{7cm} + \hbstsz \norm{\cdmat(t)}\nn \\
        & \le \quanto \norm{\dmat(t) - \ones \avgdmat(t)} + \stsz \norm{\cdmat(t)} + \hbstsz \norm{\hbmat(t)}.
    \end{align}
    The last inequality is obtained using a similar argument as in \eqref{eq:submultiplicity}. In particular, we use \(\norm{(\Wght - \identity{\dimdv})(\dmat(t) - \ones\avgdmat(t)} \le \norm{\Wght - \identity{\dimdv}}_2 \norm{\dmat(t) - \ones\avgdmat(t)}.\)
    From \eqref{eq:random} in Appendix \ref{sec:proof of lemma} and Lemma \ref{it:port_4}, \(\norm{\cdmat(t)} \le \frac{1}{\strconv}\norm{\gtmat(t) - \ones \grd \ol{\Objfunc}(\dmat(t))} + \frac{\lips}{\strconv} \norm{\dmat(t) - \ones \avgdmat(t)} + \frac{\sqrt{\agents} \lips}{\strconv} \norm{\avgdmat(t)^{\top} - \dvar^{\ast}}.\)
    Putting together everything, the expression for \(\norm{\hbmat(t)}\) in \eqref{eq:mom term} is given by
    \begin{align}\label{eq:mom term final}
        &\norm{\hbmat(t+1)} \le (\quanto + \stsz \cn )\norm{\dmat(t) - \ones \avgdmat(t)} \nn \\ &+ \frac{\stsz}{\strconv}\norm{\gtmat(t) - \ones \grd \ol{\Objfunc}(\dmat(t))} + \stsz 
        \cn \sqrt{\agents} \norm{\avgdmat(t)^{\top} - \dvar^{\ast}} \nn \\ & \hspace{5cm}+ \hbstsz \norm{\hbmat(t)}.
    \end{align}
    Substituting \eqref{eq:mom term final} back in \eqref{eq:gt aux_2}, the final expression for the gradient-tracking error is given by
    \begin{align}
        \label{eq:grad track final}
        &\norm{\gtmat(t+1) - \ones \avggtmat(t)}\nn \le \lips(\quanto + \stsz \cn)\norm{\dmat(t) - \ones \avgdmat(t)}\\& + (\specnorm + \stsz \cn) \norm{\gtmat(t) - \ones\grd \ol{\Objfunc}(\dmat(t)} + \stsz \lips \cn \sqrt{\agents}\norm{\avgdmat(t)^{\top} - \dvar^{\ast}}\nn  \\& \hspace{5cm}+ \lips \hbstsz \norm{\hbmat(t)}.
    \end{align}
    Using Lemma \ref{it:port_2}, we expand the optimality error to get
    \begin{align}\label{eq:opt err}
        &\norm{\avgdmat(t+1)^{\top} - \dvar^{\ast}}\nn\\
        & \le \underbrace{\norm{\avgdmat(t)^{\top} - \dvar^{\ast} - \stsz \frac{1}{\agents}\sum_{i=1}^{\agents} \big[\hess \objfunc_i(\dvar_i(t))\big]{\inverse} \grd \objfunc\big( \avgdmat(t)^{\top}\big)}}_{T_3} \nn\\
        & + \underbrace{\norm{\stsz \frac{1}{\agents}\sum_{i=1}^{\agents} \big[\hess \objfunc_i(\dvar_i(t))\big]{\inverse} \Big(\grd \objfunc\big(\avgdmat(t)^{\top}\big) - \gtvar_i(t)  \Big)}}_{T_4}\nn \\
        & \hspace{5cm}+ \underbrace{\norm{\hbstsz \frac{1}{\agents}\ones^{\top} \hbmat(t)}}_{T_5}.
    \end{align}
    Let us first start with \(T_3\). Observe that 
    \begin{align*}
        &T_3  \le \bigg \lVert\bigg(\identity{\agents} - \stsz \frac{1}{\agents} \sum_{i=1}^{\agents} \big[\hess \objfunc_i(\dvar_i(t))\big]{\inverse} \cdots\\&\cdots \int_0^1 \hess \objfunc\Bigl(\dvar^{\ast} + \varsigma \bigl(\avgdmat(t)^{\top} - \dvar^{\ast} \bigr)\Bigr)\odif{\varsigma} \bigg) \bigg \rVert \norm{\bigl(\avgdmat(t)^{\top} - \dvar^{\ast} \bigr)}.
    \end{align*}
    From the hypothesis of Proposition \ref{prop:lin ineq}, if \(\stsz \le \frac{\strconv}{\lips}\), then we can write
    \begin{align*}
        &0 \preceq \bigg(1 - \stsz\frac{\lips}{\strconv}\bigg)\identity{\dimdv}\preceq \identity{\agents} - \stsz \frac{1}{\agents} \sum_{i=1}^{\agents} \big[\hess \objfunc_i(\dvar_i(t))\big]{\inverse} \cdots\\&\cdots \int_0^1 \hess \objfunc\Bigl(\dvar^{\ast} + \varsigma \bigl(\avgdmat(t)^{\top} - \dvar^{\ast} \bigr)\Bigr)\odif{\varsigma} \preceq \bigg(1 - \stsz \frac{\strconv}{\lips}\bigg) \identity{\dimdv}.
    \end{align*}
    Therefore, we get
    \begin{align}\label{eq:t3}
        T_3 \le \bigg(1 - \stsz \frac{\strconv}{\lips}\bigg)\norm{\bigl(\avgdmat(t)^{\top} - \dvar^{\ast} \bigr)}.
    \end{align}
    Moreover, the term \(T_5\) can be upper bounded as:
    \begin{align}\label{eq:t5}
        T_5 &= \norm{\hbstsz \frac{1}{\agents} \sum_{i=1}^{\agents}\hbvar_i(t)} \le \hbstsz \frac{1}{\agents}\sum_{i=1}^{\agents}\norm{\hbvar_i(t)}\nn 
        \\
        & \le \hbstsz\sqrt{\frac{1}{\agents} \sum_{i=1}^{\agents} \norm{\hbvar_i(t)}^2} = \frac{\hbstsz}{\sqrt{\agents}} \norm{\hbmat(t)}. 
    \end{align}
    Now we analyze the term \(T_4\). Observe that
    \begin{align}\label{eq:t4}
        &T_4 \le \frac{\stsz}{\strconv \agents} \sum_{i=1}^{\agents}\norm{\grd \objfunc\big(\avgdmat(t)^{\top}\big) - \gtvar_i(t) } \nn \\
        & \le \frac{\stsz}{\strconv \agents} \sum_{i=1}^{\agents}\norm{\grd \objfunc\big(\avgdmat(t)^{\top}\big) + \grd \ol{\Objfunc}\big( \dmat(t)\big) - \grd \ol{\Objfunc}\big( \dmat(t)\big) - \gtvar_i(t)}\nn\\
        & \le \frac{\stsz}{\strconv \agents}\sum_{i=1}^{\agents} \norm{\gtvar_i(t) - \grd \ol{\Objfunc}\big( \dmat(t)\big)} \nn \\ & \hspace{2cm}+ \frac{\stsz}{\strconv \agents}\sum_{i=1}^{\agents} \norm{\grd \objfunc\big(\avgdmat(t)^{\top}\big) - \grd \ol{\Objfunc}\big( \dmat(t)\big)} \nn \\
        & \le \frac{\stsz}{\strconv \sqrt{\agents}}\norm{\gtmat(t) - \ones \avggtmat(t)} + \frac{\stsz\lips}{\strconv \sqrt{\agents}} \norm{\dmat(t) - \ones \avgdmat(t)} \nn\\ & \hspace{3cm}\text{(invoking Lemma \ref{it:port_3}).}
    \end{align}
    Combining \eqref{eq:t3}, \eqref{eq:t4}, and \eqref{eq:t5}, \eqref{eq:opt err} implies
    \begin{align}
        \label{eq:opt err final}
        &\sqrt{\agents}\norm{\avgdmat(t+1)^{\top} - \dvar^{\ast}} \le \stsz \cn \norm{\dmat(t) - \ones \avgdmat(t)}  \\
        &   +\frac{\stsz}{\strconv}\norm{\gtmat(t) - \ones \avggtmat(t)}  + \bigg(1 - \stsz \frac{\strconv}{\lips}\bigg)\sqrt{\agents}\norm{\bigl(\avgdmat(t)^{\top} - \dvar^{\ast} \bigr)} \nn \\
        &   \hspace{6cm}+ \hbstsz \norm{\hbmat(t)}.\nn
    \end{align}
    The final expression for the inequality \(\errvec(t+1) \le \conmat(\stsz,\hbstsz)\errvec(t),\)
    follows from \eqref{eq:con err result}, \eqref{eq:grad track final}, \eqref{eq:opt err final}, and \eqref{eq:mom term final}, completing the proof.
\end{proof}
We are now ready to state the main result of this article, which establishes a set of sufficient conditions on \((\stsz, \hbstsz)\) such that \(\spec\big(\conmat(\stsz, \hbstsz)\big) < 1\). 
\begin{theorem}
    \label{th:key}
    Consider Algorithm \ref{alg:sec_ord_comp} with its notations, and suppose that Assumptions \ref{assum:on graph} and \ref{assum:Standard assumptions} hold. Define the quantity \(\ol{\specnorm} \Let 1 - \specnorm\), and 
    \(\err \Let (\err_1, \err_2, \err_3, \err_4)\) with \(\err_i > 0\) for each \(i \in \aset[]{1,2,\ldots, \agents}\), such that 
    \begin{align}
        \label{eq:suff cond1}
        0 < \err_1 < \min\aset[\bigg]{\frac{\ol{\specnorm}\err_2}{\lips \quanto}, \frac{\err_3}{\cn^2} -  \frac{\err_2}{\strconv \cn}, \frac{\err_4}{\quanto}}. 
    \end{align}
    If \((\stsz,\hbstsz)\) satisfy
    \begin{align}
        \label{eq:parameters}
        \begin{cases}
            0 < \stsz < \min \aset[\Big]{\frac{1}{\cn}, \frac{\ol{\specnorm}\err_1}{\wt{\err}}, \frac{\ol{\specnorm} \err_2 - \lips \quanto \err_1}{\ol{\err}}, \frac{\err_4 - \quanto \err_1}{\wt{\err}}}\vspace{1.5mm}\\
            0 < \hbstsz < \min \left \{\frac{\wt{\err}}{\err_4}\Big(\frac{\ol{\specnorm}\err_2}{\wt{\err}} - \stsz \Big), \frac{\ol{\err}}{\lips \err_4}\Big( \frac{\ol{\specnorm} \err_2 - \lips \quanto \err_1}{\ol{\err}} - \stsz\Big),\cdots \right. \\ \hspace{1cm} \left. \cdots, \frac{\stsz}{\err_4} \big( \frac{\err_3}{\cn} - \cn \err_1 - \frac{\err_2}{\strconv}\big), \frac{\wt{\err}}{\err_4}\Big(\frac{\err_4 - \quanto \err_1}{\wt{\err}} - \stsz \Big)\right \},
        \end{cases}
    \end{align}
    where \(\wt{\err} \Let \cn(\err_1 + \err_3) +  \frac{\err_2}{\strconv}\) and \(\ol{\err} \Let \lips \cn (\err_1 + \err_3) +  \cn \err_2\), then \(\spec\big((\conmat(\stsz, \hbstsz)\big) < 1\), and \(\errvec(t) \xrightarrow[t \ra +\infty]{} 0\) globally, at a linear rate of \(\Bigoh\Big( \spec(\conmat(\stsz, \hbstsz))^t \Big).\)
\end{theorem}
\begin{proof}
    We start by establishing a set of sufficient conditions such that 
    \begin{equation}\label{eq:linear ineq final}
        \conmat(\stsz,\hbstsz) \err <  \err \text{ for every fixed and positive }\err>0,
    \end{equation}
    where \(\conmat(\stsz,\hbstsz)\) is defined in \eqref{eq:lin system ineq}. The proof follows by identifying an admissible vector \(\err>0\) such that \eqref{eq:linear ineq final} is satisfied, giving rise to a set of inequalities for \((\stsz, \hbstsz)\), for the fixed and admissible \(\err\). We then appeal to \cite[Corollary \(8.1.29\)]{ref:RAJ-CRJ-13} to prove the assertion. The condition \eqref{eq:linear ineq final} is analyzed row-wise, in what follows. 
    
    Let us consider the first row of \eqref{eq:linear ineq final}, which is given by the expression
    \begin{align}\label{eq:first row}
        \big(\specnorm + \stsz \cn \big) \err_1  + \frac{\stsz}{\strconv} \err_2 + \stsz \cn  \err_3 + \hbstsz \err_4 <  \err_1,
    \end{align}
    which can be rewritten as \(\hbstsz \err_4 < (1 - \specnorm)\err_1 - \stsz \Big(\cn(\err_1 + \err_3) + \frac{\err_2}{\strconv} \Big) = \ol{\specnorm}\err_1 - \wt{\err} \stsz\).
    If \((\stsz, \hbstsz)\) are picked such that 
    \begin{align}
        \label{eq:cond1}
        0 < \stsz < \frac{\ol{\specnorm}\err_1}{\wt{\err}} \text{ and }0 < \hbstsz < \frac{\wt{\err}}{\err_4}\bigg(\frac{\ol{\specnorm}\err_1}{\wt{\err}} - \stsz \bigg),
    \end{align}
    then \eqref{eq:first row} is satisfied and remains valid. We now consider the second row, which is 
    expressed as
    \begin{align}\label{eq:prefinal row2}
        \lips(\quanto + \stsz \cn)\err_1 + (\specnorm + \stsz \cn) \err_2 &+ \stsz \lips \cn  \err_3 + \lips \hbstsz \err_4 <  \err_2,
    \end{align}
    which we write as
    \begin{align*}
        \lips \hbstsz \err_4 
        & < \err_2 - (\quanto + \stsz \cn)\err_1 -(\specnorm + \stsz \cn) \err_2 - \stsz \lips \cn  \err_3 \\
        & = \Big((1 - \specnorm)\err_2
        - \lips \quanto \err_1 \Big) - \stsz \big(\lips 
        \cn (\err_1 + \err_3) + \cn \err_2\big) \\
        & = \Big(\ol{\specnorm}\err_2 - \lips \quanto \err_1\Big) - \ol{\err} \stsz. 
 \end{align*}
 implying that \(\hbstsz < \frac{(\ol{\specnorm}\err_2 - \lips \quanto \err_1) - \ol{\err} \stsz}{\lips \err_4}\). Consequently, it is sufficient to choose \((\stsz, \hbstsz)\) satisfying
 \begin{align}
     \label{eq:step-size final}
     \begin{cases}
        0 < \stsz < \frac{\ol{\specnorm}\err_2 - \lips \quanto \err_1}{\ol{\err}}, \\
        0 < \hbstsz < \frac{\ol{\err}}{\lips \err_4}\Big( \frac{\ol{\specnorm} \err_2 - \lips \quanto \err_1}{\ol{\err}} - \stsz\Big)\text{ and }
        \err_1 < \frac{\ol{\specnorm}\err_2}{\lips \quanto},
    \end{cases}
 \end{align}
 such that \eqref{eq:prefinal row2} to hold. Similarly, for the third row, we write
 \begin{align}\label{eq:prefinal row3}
     \hbstsz \err_4 < \stsz \frac{\err_3}{ \cn} - \stsz \cn \err_1 - \stsz \frac{\err_2}{\strconv}, 
 \end{align}
 from which \(\hbstsz < \frac{\stsz \frac{\err_3}{\cn} - \stsz \cn \err_1 - \stsz \frac{\err_2}{\strconv}}{\err_4} = \frac{\stsz}{\err_4} \big( \frac{\err_3}{\cn} - \cn \err_1 - \frac{\err_2}{\strconv}\big)\). Of course, one can check that \(\frac{\err_3}{\cn} -  \frac{\err_2}{\strconv} > \cn \err_1 \) is necessary for \eqref{eq:prefinal row3} to be true. Similar arguments can be repeated for the fourth row, and we directly write the sufficient conditions, as upper bounds, on the step-sizes \((\stsz, \hbstsz)\),
 \begin{align}\label{eq:row4}
    \begin{cases}
        0 < \stsz < \frac{\err_4 - \quanto \err_1}{\wh{\err}},\\
        0 < \hbstsz < \frac{\wt{\err}}{\err_4}\bigg( \frac{\err_4 - \quanto \err_1}{\wh{\err}} - \stsz\bigg) \text{ and }
        \err_1 < \frac{1}{\quanto}\err_4.
    \end{cases}
 \end{align}
 Note that \eqref{eq:cond1}, \eqref{eq:step-size final}, \eqref{eq:prefinal row3}, and \eqref{eq:row4} are sufficient conditions for \eqref{eq:linear ineq final}, for every \(\err>0\). The first assertion, \(\spec\big((\conmat(\stsz, \hbstsz)\big) < 1\), now follows from \cite[Corollary \(8.1.29\)]{ref:RAJ-CRJ-13}. For the second assertion, since \(\spec\big( \conmat(\stsz, \hbstsz)\big)< 1\), we have \( \conmat(\stsz, \hbstsz)^t \xrightarrow[t \ra +\infty]{} 0\) for \((\stsz, \hbstsz)\) satisfying \eqref{eq:suff cond1} and \eqref{eq:parameters}.  Consequently,  for every initial condition \(\errvec(0)\), \(0 \le \limsup_{t \uparrow + \infty} \errvec(t) \le \liminf_{t \uparrow + \infty} \conmat(\stsz, \hbstsz)^t \errvec(0) \xrightarrow[t \ra +\infty]{}0\) at a linear rate \(\Bigoh\big(\spec(\conmat(\stsz, \hbstsz))^t\big)\). The proof is now complete.
\end{proof}
\begin{remark}
    \label{rem:corollary}
    Following a similar line of arguments, one can obtain a feasible vector satisfying the inequality: for \(c>0\), \(\conmat(\stsz,\hbstsz)\err < \big(1  - \frac{\stsz}{c \cn} \big) \err\), such that a \emph{specific linear rate} \(\Bigoh{\Big(\big(1 - \frac{\stsz}{c \cn} \big)^t}\Big)\), where \(t\) is the iteration, can be attained. Moreover, the obtained bounds \eqref{eq:suff cond1} and \eqref{eq:parameters}, are non-conservative for positive systems. For a particular choice of an admission \(\err\), one can decompose the matrix in \eqref{eq:lin system ineq} as \(\conmat(\stsz,\hbstsz) = \conmat_0 + \stsz \conmat_1 + \hbstsz\conmat_2\). Choosing \(\err\) to be the Perron vector of \(\conmat_0\) \cite[Corollary \(8.2.6\)]{ref:RAJ-CRJ-13} allows one to establish a tight bound on \((\stsz,\hbstsz)\) from  \(\stsz \conmat_1 + \hbstsz\conmat_2 < (1 - \spec(\conmat_0))\err\).
\end{remark}
\begin{remark}
\label{rem:cn and graph}
An ill-conditioned Hessian may severely affect the step-size. Observe that if \(\cn\) is a large quantity, then from \eqref{eq:suff cond1}, \(\err_1\) is small, which implies that the step-size \(\stsz\) is small and consequently, from \eqref{eq:prefinal row3} the momentum coefficient \(\hbstsz\) must also be chosen to be small.
\end{remark}
\begin{remark}\label{rem:analytical}
    Note that the theoretical analysis of the algorithm done in this article establishes a linear rate of convergence and does not provide an analytical explanation on why acceleration is achieved by \(\algoname\). However, we believe that the analysis outlined in \cite{ref:JKW-CHL-AW-BH-22} can be extended to the distributed regime to achieve an accelerated rate locally, which requires further investigation and will be reported elsewhere. Establishing global accelerated convergence, even in the centralized case with strongly convex functions, remains an open question. An elaborate discussion may be found in \cite{ref:EG-HRF-MJ-15}.
\end{remark}


%% file: 3_the_algorithm.tex
\begin{algorithm}[!ht]
    \SetAlgoLined
    \DontPrintSemicolon
    \SetKwInOut{ini}{Initialize}
    \SetKwInOut{giv}{Data}
    \SetKwInOut{out}{Output}
    
    \giv{Tolerance \(\epsilon>0\), step-size \(\stsz\), momentum coefficient \(\hbstsz\)}
    
    \ini{\(\dmat(0)\), impose \(\gtmat(0) = \grd \Objfunc(\dmat(0))\), \(\dmat(0) = \dmat(-1)\), \(\hbmat(0) = \dmat(0) - \dmat(-1)\) 
     }
    
    \While{\(\norm{\grd \Objfunc(\dmat(t))} \le \epsilon \)}{ Update rule: The variables \((\dmat(t), \G(t), \gtmat(t), \hbmat(t))\) follow the recursions in \ref{eq:update rule}          

        \vspace{2mm}
        Update iteration index \(t \gets t+1\)
    }
    
    \out{\(\bigl( \dmat(t), \gtmat(t) \bigr)_{t=0}^{\horizon}\)}

\caption{\algoname:  An accelerated Newton-type
fully distributed optimization algorithm}
\label{alg:sec_ord_comp}
\end{algorithm}

%% file: 4_Experiments.tex
\section{Experiments}
\label{sec:experiments}
We compare \(\algoname\) with accelerated first-order algorithms --- \(\mathcal{A}\mathcal{B}m\) \cite{ref:RX-UAK-19} and Acc-DNGD-SC \cite{ref:GQ-NL-19} --- and non-accelerated algorithms --- Network-GIANT \cite{ref:AM-GS-LS-SD-23} and the first-order distributed algorithm \cite{ref:GQ-NL-17} (which we refer to as GradTrack in the sequel) --- and empirically demonstrate accelerated rate of convergence in \(\algoname\). 

Given the local datasets  \(\mathcal{D}_i \Let \aset[\Big]{\big(u^i_j,v^i_j\big) \in \Rbb^{1 \times \dimdv} \times \aset[]{-1,1}}_{j=1}^{m_i}\) available to the \(i\)-th agent, we consider a binary classification based logistic regression problem, expressed as the following optimization problem: for \(\regu>0\),
\begin{align*}
    \min_{\dvar \in \Rbb^{\dimdv}} \frac{1}{\agents}\sum_{i=1}^{\agents} \frac{1}{m_i}\sum_{j=1}^{m_i} \log\Big( 1 + \exp\big(-v^i_j(\dvar^{\top}u^i_j) \big)\Big) + \frac{\regu}{2} \norm{\dvar}^2.
\end{align*}
For the experiments we considered two undirected graphs  --- \(\wt{d}\)-regular expander graph with degree equals \(14\) and \(\specnorm = 0.3026\) (indicating dense connectivity), and an Erd\H{o}s-R\'enyi graph with edge/connection probability \(\wt{p} = 0.3\) and \(\specnorm = 0.724017\) (indicating sparse connectivity) --- with \(\agents = 20\) connected agents in both cases to model connectivity among the agents.  The Metropolis-Hastings algorithm was used to construct the consensus weight matrix, assuming that each agent can exchange data with its \(1\)-hop neighbors. The distributed classification was performed on the CovType dataset \cite{ref:DD-CG-19} after applying standard principal component analysis (PCA) to reduce the dataset further, resulting in \(566602\) sample points and \(p = 10\) features. The samples were randomly shuffled to minimize bias and distributed homogeneously among the agents. All the variables are shuffled according to Algorithm \ref{alg:sec_ord_comp}.

The hyperparameters for the algorithms considered for comparison were tuned to obtain the best performance. For all reported experiments, the regularizer \(\regu\) is selected to be \(0.05\). Table \ref{tab:prog-comp2} summarizes the different hyperparameters chosen for the experiment.  
\begin{table}[htbp]
\centering
\scriptsize 
\begin{adjustbox}{max width=0.8\linewidth} 
\begin{tblr}{
  colspec = {l |cccc},
  row{1,2} = {azure9},
  hline{1,3,Z} = {2pt},
  hline{2} = {1pt},
  column{2-5} = {c},
}
\SetTblrInner{rowsep=1pt}
 &
\SetCell[c=2]{c} \(\wt{d} = 14\) & &
\SetCell[c=2]{c} \(\wt{p} = 0.3\)  &
\\
Algorithms & \(\stsz\) & \(\hbstsz\) & \(\stsz\) & \(\hbstsz\) \\

GradTrack \cite{ref:GQ-NL-17} &
0.095 & -- &  &  \\

\(\mathcal{A}\mathcal{B}m\) \cite{ref:RX-UAK-19} &
0.18 & 0.65 & 0.2 & 0.65   \\

Acc-DNGD-SC \cite{ref:GQ-NL-19} &
0.28 & 0.7 & 0.3 & 0.7  \\

Network-GIANT \cite{ref:AM-GS-LS-SD-23} &
0.9 & -- &  & \\

\(\algoname\) &
0.15 & 0.5 & 0.13 & 0.5   \\

\hline[2pt]
\end{tblr}
\end{adjustbox}
\caption{Hyperparameters chosen for different algorithms.}
\label{tab:prog-comp2}
\end{table}

\subsubsection*{Results and discussions}
We compared the optimal solutions obtained via the distributed algorithms with the true optimal solution \(\dvar^{\ast}\) and the true optimal value \(\objfunc(\dvar^{\ast})\), computed using the centralized Newton-Raphson algorithm with backtracking line search. Figures \ref{fig:comparison} (for \(\wt{d}\)-regular expander graph) and \ref{fig:comparison_er} (for Erd\H{o}s-R\'enyi graph) show that convergence of \(\algoname\) is at least \emph{twice} faster than Network-GIANT, which is reported to be faster than the other distributed algorithms for both graphs. This observation is more prominent on densely connected graphs (\(\wt{d}\)-regular graph).
\begin{figure}[!htbp]
    \centering
    \includegraphics[scale = 0.16]{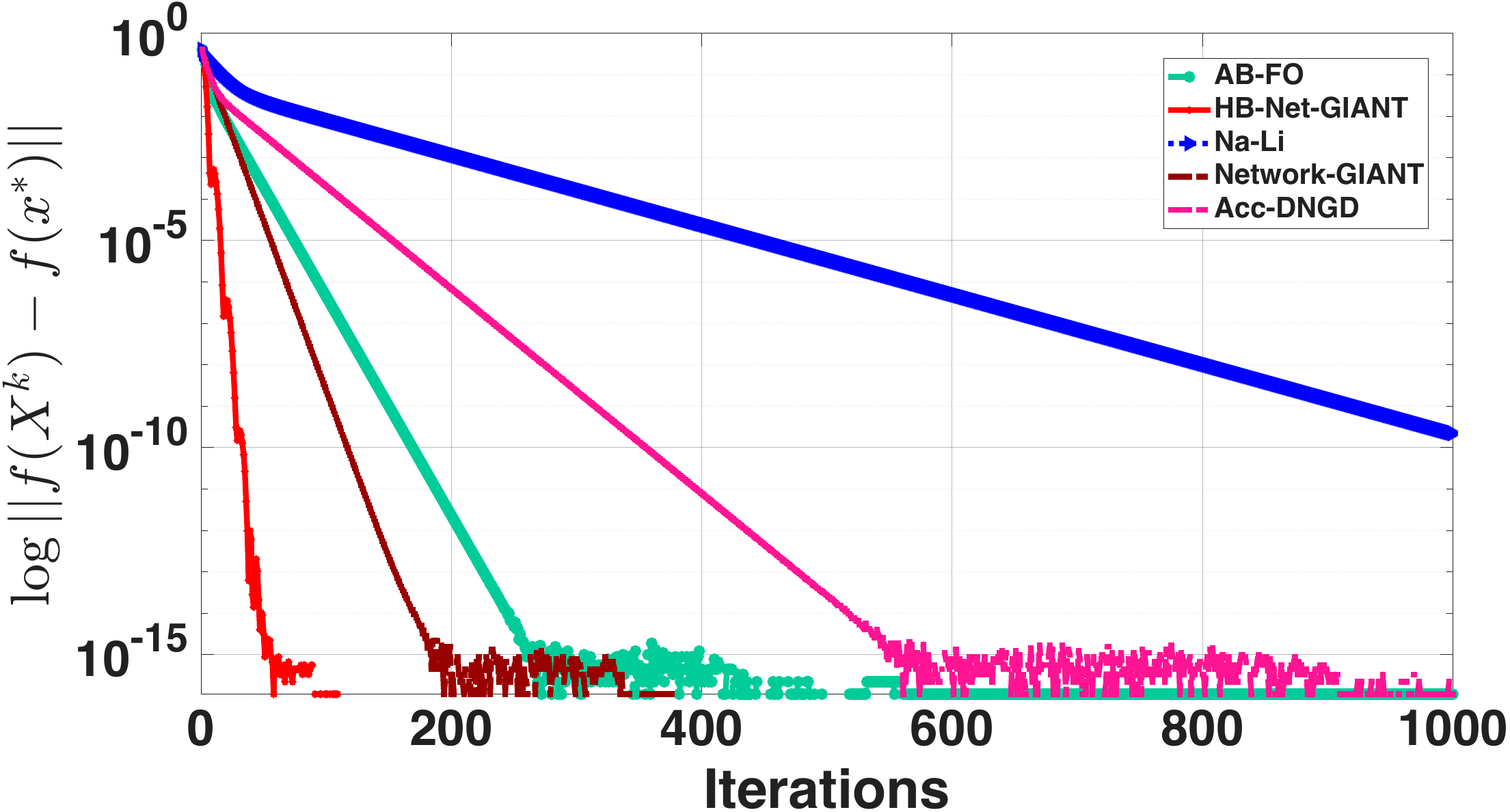}
    \caption{Comparison of \(\algoname\) with GradTrack, \(\mathcal{A}\mathcal{B}m\), Acc-DNGD-SC, and Network-GIANT for \(\wt{d}\)-regular expander graph, with \(\wt{d} = 14\).}
    \label{fig:comparison}
\end{figure}
\begin{figure}[!htbp]
    \centering
    \includegraphics[scale = 0.16]{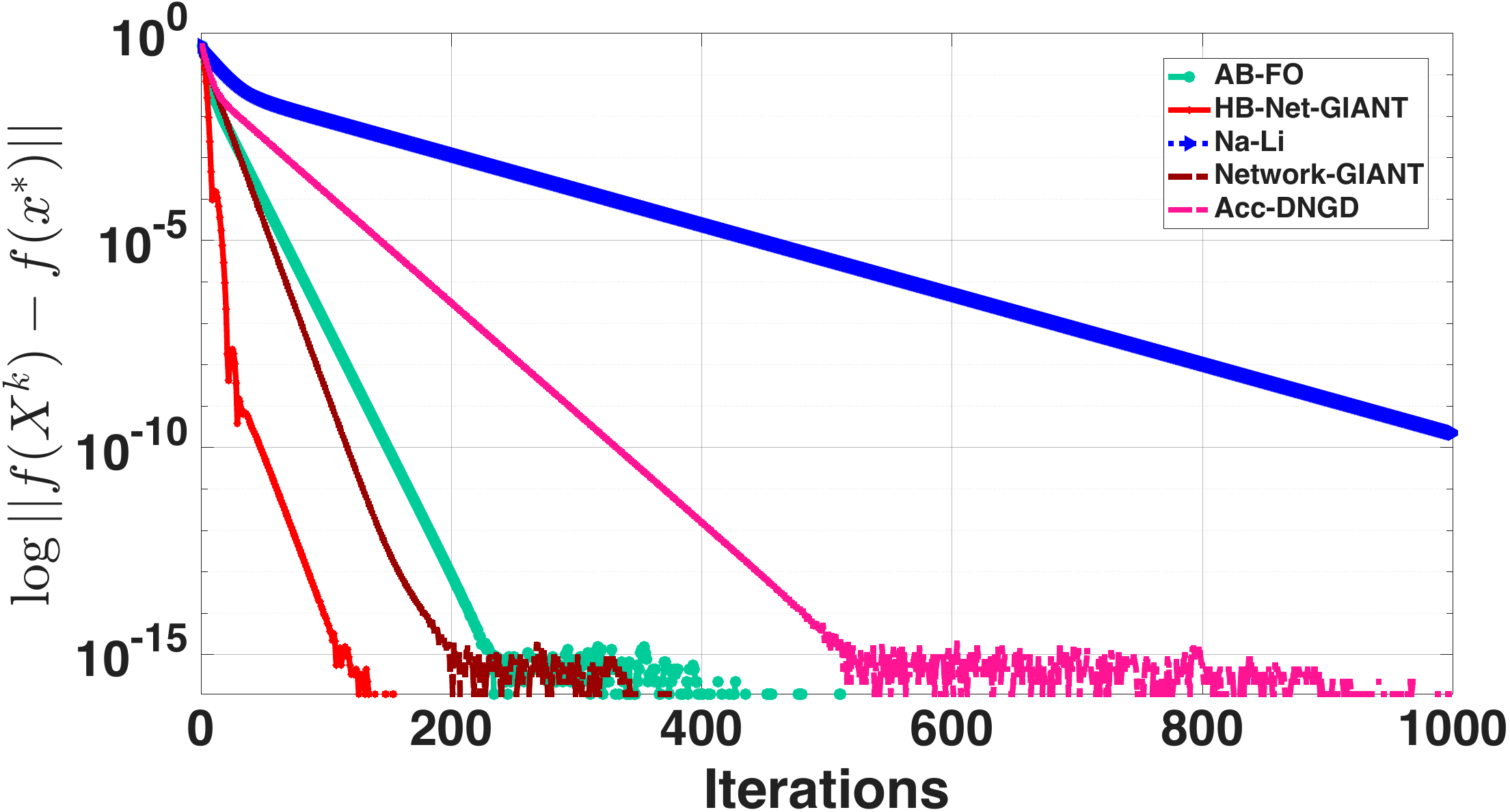}

    \caption{Comparison of \(\algoname\) with GradTrack, \(\mathcal{A}\mathcal{B}m\), Acc-DNGD-SC, and Network-GIANT for Erd\H{o}s-R\'enyi graph, with \(\wt{p} = 0.3\).}
    \label{fig:comparison_er}
\end{figure}
Theoretically, from Theorem \ref{th:key}, this rate is linear and depends on the spectral radius of the matrix \(\conmat(\stsz, \hbstsz)\), defined in \eqref{eq:lin system ineq}; specifically, the rate is given by \(\Bigoh\big(\spec(\conmat(\stsz,\hbstsz)) ^t\big)\), where \(t\) denotes the iteration index. However, as mentioned in Remark \ref{rem:analytical}, Theorem \ref{th:key} does not reveal acceleration due to the heavy-ball momentum, which requires further investigation and will be reported elsewhere.

%% file: Appendix.tex
\section{Proof of Lemma \ref{lem:portmanteuau 1}}
\label{sec:proof of lemma}

    We refer the reader to \cite[Lemma \(7\)]{ref:GQ-NL-17} for the proof of the first part.
    The proof for the assetion of Lemma \ref{it:port_2} is straightforward, and follows by multiplying both sides of the update rule for the decision variable in \eqref{eq:update rule}.
    
    For the proof of Lemma \ref{it:port_3}, we invoke Lemma \ref{it:port_1},
    \begin{align*}
        &\norm{\avggtmat(t) - \grd \objfunc\bigl( \avgdmat(t)^{\top} \bigr)} = \norm{\grd \av{\Objfunc}(\dmat(t)) - \grd \objfunc\bigl( \avgdmat(t)^{\top} \bigr)}\\
        & \le \frac{1}{\agents} \sum_{i=1}^{\agents} \norm{\grd \objfunc_i ( \dvar_i(t)) - \grd \objfunc_i\bigl( \avgdmat(t)^{\top}\bigr)} \\ 
        & \le \frac{\lips}{\agents} \sum_{i=1}^{\agents} \norm{\dvar_i(t) - \avgdmat(t)^{\top}} \le \sqrt{ \frac{\lips^2}{\agents} \sum_{i=1}^{\agents} \norm{\dvar_i(t) - \avgdmat(t)^{\top}}^2 }\\
        &= \frac{\lips}{\sqrt{\agents}} \norm{\dmat(t) - \ones \avgdmat(t)},
    \end{align*}
    completing the proof for the first part of Lemma \ref{it:port_3}. The second part follows a similar argument:
    \begin{align*}
        &\norm{\grd \Objfunc\big(\dmat(t+1) \big) - \grd \Objfunc\big(\dmat(t) \big)} \\
        &\le \sqrt{\sum_{i=1}^{\agents} \norm{\grd \objfunc_i(\dvar_i(t+1)) - \grd \objfunc_i(\dvar_i(t))}^2}\\
        & \le \sqrt{\lips^2\sum_{i=1}^{\agents}\norm{\dvar_i(t+1) - \dvar_i(t)}^2} \text{ (from Assumption \ref{assum:Standard assumptions})}\\
        & \le \lips  \norm{\dmat(t+1) - \dmat(t)}.
    \end{align*}
    
    We now focus on Lemma \ref{it:port_5}. Expanding the terms, we have the following set of inequalities:
    \begin{align*}
        &\norm{\grd \av{\Objfunc} \big(\ones \avgdmat(t) \big) - \grd \objfunc\big(\dvar^{\ast}\big)} \\
        & \le \frac{1}{\agents}\sum_{i=1}^{\agents} \norm{\grd \objfunc_i\big(\avgdmat(t)^{\top} \big) - \grd \objfunc_i(\dvar^{\ast})} \\
        & \stackrel{(\ddagger)}{\le} \frac{\lips}{\agents}\sum_{i=1}^{\agents} \norm{\avgdmat(t)^{\top} - \dvar^{\ast}} = \lips \norm{\avgdmat(t)^{\top} - \dvar^{\ast}},
    \end{align*}
    where \((\ddagger)\) follows from Assumption \ref{assum:Standard assumptions}, therefore, proving the assertion.

    Let us now focus on proving the assertion of Lemma \ref{it:port_4}. For each iteration \(t\), expanding the square in \(\norm{\cdmat(t) - \ones \avgcdmat(t)}^2 \) and simplifying further, we get
    \begin{align}\label{eq:random}
        &\norm{\cdmat(t) - \ones \avgcdmat(t)}^2 = \norm{\cdmat(t)}^2 + \agents \norm{\frac{1}{\agents} \ones^{\top} \cdmat(t)}^2 \nn \\ & \quad - 2 \sum_{i=1}^{\agents} \inprod{\big[\hess \objfunc_i \big(\dvar_i(t)\big)\big]{\inverse} \gtvar_i(t)}{\frac{1}{\agents} \ones^{\top} \cdmat(t)}\nn\\
        & = \norm{\cdmat(t)}^2 + \agents \norm{\frac{1}{\agents} \ones^{\top} \cdmat(t)}^2   - 2 \agents \underbrace{\inprod{\frac{1}{\agents} \ones^{\top} \cdmat(t)}{\frac{1}{\agents} \ones^{\top} \cdmat(t)}}_{\teL \norm{\frac{1}{\agents} \ones^{\top} \cdmat(t)}^2} \nn\\
        & = \norm{\cdmat(t)}^2 - \agents \norm{\frac{1}{\agents} \ones^{\top} \cdmat(t)}^2 \le \norm{\cdmat(t)}^2 \nn \\ 
        & = \sum_{i=1}^{\agents} \norm{\big[\hess \objfunc_i \big(\dvar_i(t) \big)\big]{\inverse} \gtvar_i(t)}^2 \le \frac{1}{\strconv^2} \norm{\gtmat(t)}^2 \; \text{(from \eqref{eq:strongly_convex_Lips})},
    \end{align}
    where
    \begin{align*}
        &\norm{\gtmat(t)} = \bigg \lVert \gtmat(t) - \ones \grd \ol{\Objfunc}(\dmat(t)) + \ones \big(\grd \ol{\Objfunc}(\dmat(t)) \\ &  \qquad \qquad- \grd \ol{\Objfunc}(\ones\avgdmat(t))\big)  + \ones \big(\grd \ol{\Objfunc}(\ones\avgdmat(t)) - \grd \objfunc(\dvar^{\ast}) \big) \bigg \rVert \\
        & \le \norm{\gtmat(t) - \ones \grd \ol{\Objfunc}(\dmat(t))} + \sqrt{\agents}\norm{\grd \ol{\Objfunc}(\dmat(t)) - \grd \ol{\Objfunc}(\ones\avgdmat(t))} \\ & \hspace{2.5cm} + \sqrt{\agents} \norm{\grd \ol{\Objfunc}(\ones\avgdmat(t)) - \grd \objfunc(\dvar^{\ast}) }\\
        & \le \norm{\gtmat(t) - \ones \grd \ol{\Objfunc}(\dmat(t))} + \sqrt{\agents}\frac{\lips}{\sqrt{\agents}} \norm{\dmat(t) - \ones \avgdmat(t)} \\ & \hspace{4cm}+ \sqrt{\agents} \lips \norm{\avgdmat(t)^{\top} - \dvar^{\ast}} \\ & \hspace{2cm}\text{ (invoking Lemma \ref{it:port_3} and \ref{it:port_5})}.
    \end{align*}
    Therefore, the resulting expression for the upper bound of \(\norm{\cdmat(t) - \ones \avgcdmat(t)}\) is given by
    \begin{align*}
        &\norm{\cdmat(t) - \ones \avgcdmat(t)} \le \frac{1}{\strconv}\norm{\gtmat(t)}\\
        & \hspace{1cm}\le \frac{1}{\strconv}\norm{\gtmat(t) - \ones \grd \ol{\Objfunc}(\dmat(t))} + \frac{\lips}{\strconv} \norm{\dmat(t) - \ones \avgdmat(t)} \\ & \hspace{4.5cm} + \sqrt{\agents} \frac{\lips}{\strconv} \norm{\avgdmat(t)^{\top} - \dvar^{\ast}}.  
    \end{align*}
    The proof is now complete.

    For the proof of Lemma \ref{it:port_6}, we employ an argument similar to the proof of Lemma \ref{it:port_4}, and write \(\norm{\hbmat(t) - \frac{1}{\agents} \ones \ones^{\top} \hbmat(t)}^2 = \norm{\hbmat(t)}^2 + \agents \norm{\avghbmat(t)}^2 - 2 \agents \norm{\avghbmat(t)}^2 \le \norm{\hbmat(t)}^2\).
    This completes the proof of Lemma \ref{it:port_6}.